\documentclass[a4paper,12pt]{article}
\usepackage[cp1251]{inputenc}
\usepackage[ukrainian, russian, english]{babel}
\usepackage{amsmath, amsthm, amsfonts, amssymb}

\theoremstyle{plain}
\newtheorem{thm}{Theorem}
[section]
\newtheorem{lem}{Lemma}
[section]
\makeatletter
\@addtoreset{equation}{section}
\makeatother
\makeatletter
\@addtoreset{thm}{section}
\makeatother
\makeatletter
\@addtoreset{lem}{section}
\makeatother

\begin{document}
\large

\begin{center}
{\Large\bf
On a stationary random knot}\\
Andrey A.Dorogovtsev\\[10 pt]
Institute of mathematics\\
National Academy of Sciences of Ukraine\\
andrey.dorogovtsev@gmail.com

\end{center}

{\bf Abstract}

It this article the construction of a stationary random knot is proposed. The corresponding smooth random curve has no self-intersections in deterministic moments of time and changes its topological type at random moments.



\section*{Introduction}

In this article the construction of a stationary process of closed smooth random curves in ${\mathbb R}^3$ is proposed. The main advantage of this construction consists of  two features. Firstly, for fixed deterministic moment  of time obtained random curve has no self-intersections. Thus it is indeed a random knot. Secondly, it changes topological type with time. The interest to such an object arose in the mathematical modeling of linear polymers. There exists a well-developed theory of polymer models based on the modeling random walk or Brownian motion in Euclid space (see for example [1-3]). One of the essential problems here  is the existence of the self-intersection points which are not present in the real linear polymer. To overcame this difficulty some counting and penalization of self-intersections is provided [4 -- 11]. Another important feature of models based on the Brownian motion is the nondifferentiability of trajectories. This causes problems with geometric properties of the obtained random curve. It leads to the use of the Haussdorff measure or consideration of self-intersection renormalized local times as the Minkowski functionals in tube formula [11]
 Such complicated approach works well for Brownian motion or for Gaussian process of similar structure [9 -- 11]
 But if we want to consider a moving curve it is not natural to assume that a curve that changes its shape with time can still be described as Brownian motion.

 In this paper we construct the stationary smooth closed random curve using a stochastic differential equation with interaction. Such kind of equations were introduced by the author [13 -- 17] in order to describe the motion of large systems of interacting particles together with their mass distribution. It occuers that if one takes as initial mass distribution the visitation measure of a certain curve then equation can be interpreted as a rule of motion of this curve in the space. This idea has one disadvantage. Namely, all obtained curves are homotopic because the flow generated by the equation with reasonably good coefficients consists of homeomorphisms. Hence the moving curve will keep its topological type. To allow the  changing of the type (which also means that in some moments curve must have intersections) we consider not the solution to equation with interaction but its image under the random mapping from ${\mathbb R}^3$ to ${\mathbb R}^3.$ This random mapping is choosen to be centered Gaussian with smooth rapidly decreasing covariance. Obtained random curve has interesting properties. On the one hand it has no self-intersections, but on the other hand it changes its topological type. This, in particular, means that there are self-intersections at random moments of time. It must be noted that equations with interaction were used for description of the motion of random curves or surfaces in 
[18], 
 but in that article  non-smooth surfaces where considered. Also, they remain to be homotopic to initial surface and the self-intersection local times were investigated as their geometrical characteristics. In contrast to the mentioned work, the stationary random knot, which is built here, is smooth and changes its topological type.To present this construction, the paper is devided onto three sections. In the first section we consider the image of the smooth deterministic closed curve under the random mapping. This random mapping itself is a rotation and translation invariant centered Gaussian random field. As an example of deterministic curve we consider the circle on the plane. It is proved in the first section, that obtained random curve has no self-intersections and can be of arbitrary topological type with positive probability. Of cause, all statements from this section remain valid for an arbitrary initial closed curve, not just a circle. In the second section  equations with interaction are applyed to the description of the motion of the curve. Here the image of the  moving measure under the random mapping is introduced and its mean characteristics are included into coefficients of equation. Last section contains the construction of the stationary random knot as an application of the statements from the previous two sections.

\section{Curves under stationary random mappings}
\label{section1}
In this section we consider the image of the smooth curve on the plane under stationary random mapping acting from ${\mathbb R}^2$ to ${\mathbb R}^3$ and discuss its topological properties. The construction will be as follows. Let $\vec{\xi}=(\xi_1, \xi_2, \xi_3): {\mathbb R}^2\to{\mathbb R}^3$ be a centered Gaussian random field with independent and identically distributed coordinates with the  covariance
\begin{equation}
\label{eq1.1}
E\xi_1(\vec{u})\xi_1(\vec{v})=\Gamma(\vec{u}-\vec{v})=e^{-\|\vec{u}-\vec{v}\|^2}.
\end{equation}
Since the covariance \eqref{eq1.1} is infinitely  differentiable, then due to Gaussianity,  $\vec{\xi}$ has an infinitely  differentiable on ${\mathbb R}^2$ modification. Define a random curve $\gamma$ in ${\mathbb R}^3$ using parametrization $\vec{\gamma}(t)=\vec{\xi}(\cos t, \sin t), \ t\in[0; 2\pi].$ So, $\gamma$ is an image of the unit circle $S$ on the plane under the mapping $\vec{\xi}.$
 Here we prove the following statements about the random curve $\gamma.$
\begin{thm}
\label{thm1}
With probability one $\gamma$ has no points of self-intersection, namely
$$
P\{\exists \, t, s\in[0; 2\pi), t\ne s: \ \ 
\vec{\gamma}(t)=\vec{\gamma}(s)\}=0.
$$
\end{thm}
\begin{thm}
\label{thm2}
Let $\vec{f}\in C^2([0; 2\pi])$ be a parametrization of a smooth closed curve in ${\mathbb R}^3$ without self-intersections. Then the set of $\omega\in\Omega$ such that $\vec{\xi}(\gamma_0)$ is homotopic to $\vec{f}$ is a random event of positive probability.
\end{thm}
We prove first some auxiliary lemmas.
Denote 
$$\zeta_1(t)=\xi_1(\cos t, \sin t), \  t\in [0; 2\pi].$$
\begin{lem}
\label{lem1}
If $\mu_\gamma$ is the distribution of the Gaussian random function $\zeta_1(t), t\in[0; 2\pi]$ in the space $C_p([0; 2\pi])$ of continuous periodic functions, then
$$
\rm{supp}\, \mu_\gamma= C_p([0; 2\pi]).
$$
\end{lem}
\begin{proof}
 Since $\mu_\gamma$ is Gaussian measure, then $\rm{supp}\, \mu_\gamma$ is a closed linear subspace in $C_p([0; 2\pi]).$ Suppose, that $\rm{supp}\, \mu_\gamma$ is strictly less then $C_p([0; 2\pi]).$ Then by the Hahn-Banach theorem there exists non-trivial continuous linear functional $\varkappa\in C_p([0; 2\pi])^*$ such that its kernel contains $\rm{supp}\, \mu_\gamma.$ The functional $\varkappa$ can be interpreted as a signed measure on $S.$ Then for the random field $\xi_1$
$$
\int_S\xi_1(\vec{v})\varkappa(d\vec{v})=0.
$$
Hence
\begin{equation}
\label{eq1.3}
\forall \ \vec{u}\in{\mathbb R}^2: \ \ 
\int_Se^{-\|\vec{u}-\vec{v}\|^2}\varkappa(d\vec{v})=E\vec{\xi}_1(\vec{u})\int_S\xi_1(\vec{v})\varkappa(d\vec{v})=0.
\end{equation}
 
 Then
 \begin{equation}
 \label{eq1.4}
 \forall \ \vec{\lambda}\in{\mathbb R}^2:  \ \ \hat{\Gamma}_{\frac{1}{4}}(\vec{\lambda})\hat{\varkappa}(\vec{\lambda})=0,
 \end{equation}
 where $\hat{\Gamma}_{\frac{1}{4}},\hat{\varkappa}$ are the Fourier transforms of ${\Gamma}_{\frac{1}{4}}$ and $\varkappa.$ It follows from \eqref{eq1.4} that $\varkappa$ can not be nontrivial and  our assumption about $\rm{supp}\,\mu_\gamma$ was wrong. Lemma is proved.
\end{proof}

Define on ${\mathbb R}^3$ the following family of functions approximating $\delta$-function at $\vec0$
$$
h_\varepsilon(\vec{x})=\frac{1}{\frac{4}{3}\pi\varepsilon^3}{1\!\!\,{\rm I}}_{B(\vec{0}, \varepsilon)}(\vec{x}), \ \vec{x}\in{\mathbb R}^3, \ \varepsilon>0.
$$
Here
$$
B(\vec{0}, \varepsilon)=\{\vec{x}: \|\vec{x}\|<\varepsilon\}.
$$
\begin{lem}
\label{lem1.2}
There exists a finite limit
$$
\lim_{\varepsilon\to0}\int^b_a\int^d_cEh_\varepsilon(\vec{\gamma}(t_1)-\vec{\gamma}(t_2))dt_1dt_2=
$$
$$
=\int^b_a\int^d_c\frac{dt_1dt_2}
{(2\pi(2-2\Gamma(\vec{u}(t_2)-\vec{u}(t_1))))^{3/2}},
$$
where $0\leq a<b<c<d<2\pi, \ \vec{u}(t)=(\cos t,\ \sin t).$
\end{lem}

Proof of the lemma easily follows from the Lebesgue dominated convergence theorem.

As a consequence of Lemma \ref{lem1.2} one can obtain a proof of Theorem \ref{thm1}.

\begin{proof}[Proof of Theorem \ref{thm1}]
It is enough to check that the pieces of $\gamma,$ \ $\vec{\gamma}([a; b])$ and $\vec{\gamma}([c; d]),$ have no intersections with probability one for $0\leq a<b<c<d<2\pi.$ First note, that the set $A$ of $\omega\in\Omega$ such that
\newline $\exists \ t_1\in[a; b], t_2\in[c;d]:$
\begin{equation}
\label{eq1.5}
\vec{\gamma}(t_1)(\omega)=\vec{\gamma}(t_2)(\omega)
\end{equation}
is measurable. Indeed, the restrictions of $\vec{\gamma}(\cdot)$ on $[a; b]$ and $[c; d]$ are continuous functions. Consequently, the opposite property to \eqref{eq1.5} looks like
\newline $\exists \ n\geq1: \forall \ r_1\in{\mathbb Q}\cap[a;b], \ \forall \ r_2\in{\mathbb Q}\cap[c;d]:$
\begin{equation}
\label{eq1.6}
\|\vec{\gamma}(r_1)(\omega)-\vec{\gamma}(r_2)(\omega)\|>\frac{1}{n}.
\end{equation}
Suppose now, that $P(A)>0.$
 Take $\omega\in A$ and consider
\begin{equation}
\label{eq1.6'}
\int^b_a\int^d_c h_\varepsilon(\vec{\gamma}(t_1)(\omega)-\vec{\gamma}(t_2)(\omega))dt_1dt_2.
\end{equation}
 If there exists a pair $(t^0_1, t^0_2)\in[a;b]\times[c;d]$ such that
 $$
 \vec{\gamma}(t^0_1)(\omega)=\vec{\gamma}(t^0_2)(\omega),
 $$
 then, due to differentiability of $\vec{\gamma}(\cdot)(\omega)$ at these points, there exists positive constant (depending on $\omega$) $C(\omega)>0$ such that integral in 
\eqref{eq1.6'}
satisfies the inequality
\newline $\exists \ \varepsilon_0: \forall \ \varepsilon<\varepsilon_0:$
\begin{equation}
\label{eq1.7}
\int^b_a\int^d_c h_\varepsilon(\vec{\gamma}(t_1)(\omega)-\vec{\gamma}(t_2)(\omega))dt_1dt_2\geq C(\omega)\frac{1}{\varepsilon}.
\end{equation}

Together with the Fatou lemma \eqref{eq1.7} and the hypothesis $P(A)>0$ contradict to the statement of  Lemma \ref{lem1.2}. Consequently, our assumption that $P(A)>0$ was wrong. The theorem is proved.
\end{proof}

 It was proved, that the random curve $\vec{\xi}(\gamma_0)$  with probability one has no self-intersection. So $\vec{\xi}(\gamma_0)$ is a random knot. Now let us prove Theorem \ref{thm2}.
\begin{proof}
We start with a purely deterministic construction. For $n\geq1$ consider the uniform partition $0=t_0<\ldots<t_n=2\pi$ (i.e. $t_k=\frac{k2\pi}{n}$). Define $\vec{f}_n$ as a polygonal interpolation of $\vec{f}$ with its values at $t_0, \ldots, t_n.$ Then there exists $n_0$ such, that for $n\geq n_0$ \ $\vec{f}_n$ is homotopic to $\vec{f}$ (i.e. the polygonal line and the curve described by $\vec{f}_n$ and $\vec{f}$ define the knot of the same topological type). It was proved in \cite{20}. Consequently, to check the homotopy between $\vec{\xi}(\gamma_0)$ and $\vec{f},$ it is enough to check the homotopy of corresponding polygonal lines. To check that two polygonal lines are homotopic it is enough to consider flat diagrams and try to get one from another by the finite sequence of Reidemeister moves \cite{22}. Eeach moving is applied to a certain crossroad. So, one have to consider finite sequences of Reidemeister moves and crossroads. Hence the occurrence of homotopy between random and fixed polygonal lines is a random event. 

Now we will prove that this event has a positive  probability. Here we need  auxilary notations and statement. Denote    by $C^2_p([0; 2\pi])$ the set of all functions which are two times continuously differentiable and
$2\pi$-periodic. Let $\eta(t)=\xi(\vec{\gamma}(t)), t\in[0; 2\pi]$ be a Gaussian random process defined with the help of the centered Gaussian random field $\xi,$ which is equidistributed with $\xi_i, i=1,2,3. $ Let $\mu$ be the distribution of $\eta''$ in the space of all functions
$$
M=\{g: g=f'',  \ f\in C^2_p([0; 2\pi])\}
$$
with uniform distance.
\begin{lem}
\label{lem3}
$$
\rm{supp}\,\mu=M.
$$
\end{lem}

\begin{proof} The proof is similar to the proof of Lemma \ref{lem3},
but we include it for the  convenience. Since $\mu$ is centered Gaussian measure in $M,$ then it is enough to check, that there is no such linear continuous functional $\varkappa$ on $M,$ that
$$
\exists \ g\in M: \langle g, \varkappa\rangle\ne0,
$$
$$
\forall \ h\in\rm{supp}\, \mu: \langle h, \varkappa\rangle=0.
$$
To check this, suppose that such $\varkappa$ exsists and get contradiction. Indeed, since $M$ can be viewed as a subspace of $C([0; 2\pi]),$ then $\varkappa$ can be identified with a finite signed measure on $[0; 2\pi].$ Correspondingly for random process $\eta$ one can write
$$
\int^{2\pi}_0\eta(t)''\varkappa(dt)=0.
$$
Then
$$
\forall \ \vec{u}\in{\mathbb R}^2: \ \int^{2\pi}_0(E\xi(\vec{u})\eta(t))''\varkappa(dt)=0.
$$
Hence
$$
\forall \ \vec{u}\in{\mathbb R}^2: \ \int^{2\pi}_0\Gamma(\vec{u}-\vec{\gamma}(t))''\varkappa(dt)=0,
$$
where as before
$$
\Gamma(\vec{u}-\vec{v})=e^{-\|\vec{u}-\vec{v}\|^2}
$$
is a covariance function of the random field $\xi.$ Let us define the generalized function $\widetilde{\varkappa}$ on ${\mathbb R}^2$ as follows\newline
$\forall \varphi\in S({\mathbb R}^2):$
$$
\langle\varphi, \widetilde{\varkappa}\rangle:=\int^{2\pi}_0\varphi(\vec{\gamma}(t))''\varkappa(dt).
$$
Then for the convolution $\Gamma*\widetilde{\varkappa}$ we have
$$
\Gamma*\widetilde{\varkappa}\equiv0.
$$
Consequently, the Fourier transform
$$
\widehat{\Gamma*\widetilde{\varkappa}}=\widehat{\Gamma}\cdot\widehat{\widetilde{\varkappa}}\equiv0.
$$
Since
$$
\forall \ \vec{\lambda}\in{\mathbb R}^2: \ \widehat{\Gamma}(\vec{\lambda})\ne0,
$$
then
$$
\widehat{\widetilde{\varkappa}}\equiv0.
$$
It means, that the generalized function $\widetilde{\varkappa}$ equals zero, Consequently, \newline $\forall \varphi\in S({\mathbb R}^2):$
$$
\int^{2\pi}_0\varphi(\vec{\gamma}(t))''\varkappa(dt)=0.
$$
It remains to note, that the set
$$
\{\varphi(\vec{\gamma}(\cdot))'': \ \varphi\in S({\mathbb R}^2)\}
$$
is dence in $M$ with respect to uniform distance.  Finally our measure $\varkappa$ is such, that\newline
$\forall \ g\in M:$
$$
\int^{2\pi}_0g(t)\varkappa(dt)=0.
$$
This means, that related functional $\varkappa$ on $M$ is zero. Obtained contradiction proves the statement of the lemma.

\end{proof}

Now let us back to the main statement. Consider a curve $F_0,$ which is parametrized by the infinitely differentiable periodic function $\vec{f}: \ [0; 2\pi]\to{\mathbb R}^3$ with the following property of its projection $\vec{f}_*$ on the first two coordinates. For all $\delta\in(0;\pi)$ there exists $\varepsilon_1>0$ such that \newline
$\forall \ t_1, t_2\in[0;2\pi], \ \min(|t_1-t_2|,  2\pi-|t_1, t_2|)>\delta:$
$$
\|\vec{f}_*(t_1)-\vec{f}_*(t_2)\|>\varepsilon_1.
$$
Consider another function $\vec{g}: [0; 2\pi]\to{\mathbb R}^3$ which is periodic, infinitely differentiable and such that for some positive $\varepsilon$
$$
\max_{[0; 2\pi]}\|\vec{f}''(t)-\vec{g}''(t)\|<\varepsilon.
$$
Due to periodicity the following relationships hold
$$
f(t)=c_1+\frac{t}{2\pi}\int^{2\pi}_0sf''(s)ds+\int^t_0(t-s)f''(s)ds,
$$
$$
g(t)=c_2+\frac{t}{2\pi}\int^{2\pi}_0sg''(s)ds+\int^t_0(t-s)g''(s)ds.
$$
Suppose that a planar curve parametrized by $\vec{f}_*$ has no self-intersections. Prove that the same holds for function $\vec{g}_*.$  Consider new function $\vec{\widetilde{f}}$ defined as
$$
\vec{\widetilde{f}}(t)=\vec{c}_2+\frac{t}{2\pi}\int^{2\pi}_0s{\vec{f}}''(s)ds+\int^t_0s{\vec{f}}''(s)ds.
$$
The curves parametrized by the functions $\vec{f}$ and $\vec{\widetilde{f}}$ are connected by parallel translation on $\vec{c}_2-\vec{c}_1$ and, consequently have the same type.
So, it is enough to compare $\vec{\widetilde{f}}$ and $\vec{g}.$ Due to the choice of $\vec{g}$ one has
$$
\max_{[0; 2\pi]}\|\vec{\widetilde{f}}(t)-\vec{g}(t)\|\leq 2\pi^2\varepsilon,
$$
$$
\max_{[0; 2\pi]}\|\vec{\widetilde{f}}'(t)-\vec{g}'(t)\|\leq 2\pi^2\varepsilon.
$$
Consider projections of $\vec{\widetilde{f}}$ and $\vec{g}$ onto the first two coordinates. Let us prove that projection of $\vec{g}$ parametrizes curve without self-intersections. Suppose that for different $t_1, t_2\in[0; 2\pi], t_1<t_2$
$$
\vec{g}_*(t_1)=\vec{g}_*(t_2).
$$
Then one can find $s_1, s_2\in[t_1; t_2]$ such, that the angle between $\vec{g}_*'(s_1)$ and $\vec{g}_*'(s_2)$ equals $\pi.$  Also note, that due to the condition on $\vec{f}_*,$ if $2\pi\varepsilon^2<\varepsilon_1,$ then
$$
t_2-t_1<\delta.
$$
Hence
$$
\|\vec{g}_*'(s_2)-\vec{g}_*'(s_1)\|\leq \delta\max_{[0; 2\pi]}\|\vec{g}_*''(s)\|.
$$
But
$$
\max_{[0; 2\pi]}\|\vec{g}_*'(s)-\vec{\widetilde{f}}'(s)\|\leq 2\pi^2\varepsilon.
$$
Consequently,
$$
2R-4\pi^2\varepsilon\leq \delta\max_{[0; 2\pi]}\|\vec{g}_*''(s)\|\leq
$$
$$
\leq\delta(\max_{[0; 2\pi]}\|\vec{f}_*''(s)\|+\varepsilon).
$$
Finally
$$
2R\leq\delta(\max_{[0; 2\pi]}\|\vec{f}_*''(s)\|+\varepsilon)+4\pi^2\varepsilon.
$$
Now it remains to choose $\delta$ in a such way, that
$$
\delta(\max_{[0; 2\pi]}\|\vec{f}_*''(s)\|<\frac{1}{2}R,
$$
then take such $\varepsilon>0,$ which satisfies inequalities
$$
2\pi\varepsilon^2<\frac{1}{2}\varepsilon_1, \ \ \varepsilon\delta+4\pi\varepsilon^2<\frac{1}{2}R.
$$
With such a choice $\varepsilon, \delta$ we can not find such pair $t_1\ne t_2,$ that $\vec{g}_*(t_1)=\vec{g}_*(t_2).$  Now it follows from these consideration and previous lemma, that with positive probability the random curve $\gamma$ is a trivial knot. To prove, that with positive probability $\gamma$ can be a nontrivial knot let us consider the function $\vec{f}_*: [0; 2\pi]\to{\mathbb R}^2$ with the same as above properties. Bat, in contrast to previous case, suppose, that the curve, which is parametrized by $\vec{f}_*$ has a finite number of self-intersection points. Denote by $t_1<s_1, \ldots, t_n<s_n$ corresponding values of parameters. Take positive $\delta$ such that $\delta$-neighbourhoods of $t_1, \ldots, t_n,\ldots, s_n$ have no intersections. Consider a function $f_3\in C^2_p([0; 2\pi])$ such that $|f_3|>c>0$ on all $\delta$-neighbourhoods of $t_1, \ldots, t_n, s_n.$ Then, it can be checked similarly to the previous considerations, that if the function $\vec{g}\in C^2_p([0; 2\pi], {\mathbb R}^3)$ approximates $\vec{f},$ then the knots parametrized by $\vec{f}$ and $\vec{g}$ have the same to pological type. This completes the proof of the statement, that the random curve $\gamma$ with positive probability has an arbitrary topological type.

\end{proof}

\section{Moving random curve in ${\mathbb R}^3$}
\label{section2}
In this section we propose a model of a moving random curve in ${\mathbb R}^3$ based on equation with interaction
 [14 -- 18] and random map $\vec{\xi}$ from the previous section. An equation with interaction is an equation of the following kind
\begin{equation}
\label{eq2.1}
d\vec{x}(\vec{u},t)=\vec{a}(\vec{x}(\vec{u},t), \mu_t)dt+\int_{{\mathbb R}^d}b(\vec{x} (\vec{u},t),\mu_t, \vec{}p)\vec{W}(d\vec{p}, dt),
 \end{equation}
 $$
 \vec{x}(\vec{u},0)=\vec{u}, \ \vec{u}\in{\mathbb R}^d, \ \mu_t=\mu_0\circ \vec{x}(\cdot, t)^{-1}.
 $$
 Here $\mu_0$ is the initial mass distribution of the system. $\vec{W}$ is an ${\mathbb R}$-valued Wiener sheet on ${\mathbb R}^d\times[0; +\infty).$ $\vec{W}$ has independent coordinates, each of which is a centered Gaussian random measure on ${\mathbb R}^d\times[0; +\infty)$ with independent values on disjoint sets and structural measure, which is equal to Lebesgue measure. $\vec{W}$ plays the role of an outer random media, which perturbs the motion of particles. The particles start from every point of the space. Infinitesimal increments of the particle trajectory depend not only on the position of the particle, but also on the mass distribution of all particles. Such form of the equation allows to consider  infinite systems of interacting particles in random media. Note, that the initial mass distribution is not necessary discrete. It can be continuous and even have density with respect to the Lebesgue measure. For us will be important the case when $\mu_0$ is a visitation measure of a certain smooth closed curve $\gamma_0.$ Namely, suppose, that $\vec{f}\in C^2_p([0; 1], {\mathbb R}^3)$ and $\gamma_0$ is parametrised by $\vec{f}.$ Then define $\mu_0$ as follows
$$
\mu_0(\Delta)=\int^1_0{1\!\!\,{\rm I}}_\Delta(\vec{f}(t))dt.
$$
Note, that $\mu_0$ is a probability measure on ${\mathbb R}^3$ depending on the parametrization of $\gamma_0.$ If we start from such $\mu_0,$ then at the moment $t>0$ \ $\mu_t$ will be a visitation measure of the curve $\gamma_t=\vec{x}(\gamma_0, t),$ which is parametrized by the function $\vec{x}(\vec{f}(\cdot), t).$ In a such way equation \eqref{eq2.1} describes the evolution of the curve. However such model has an essential disadvantage. If the coefficients of the  equation have two bounded continuous derivatives with respect to a spatial variable, then the solution $\vec{x}$ with probability one is jointly continuous with respect $\vec{u}$ and $t$ and diffeomorphic as a map from ${\mathbb R}^3$ to ${\mathbb R}^3$ for a fixed $t.$ This means that $\gamma_0$ is homotopic to $\gamma_t.$ Or, that the curve $\gamma_t$ has the same topological type as $\gamma_0.$ In order to get around such difficulty, we consider in this article not the evolution of $\gamma_t$ itself, but its image in ${\mathbb R}^3$ under the mapping $\vec{\xi}.$ To present the corresponding equation, we need to specify the coefficients, which depend on the visitation measure of $\gamma_t.$ Consider the function
$$
\vec{h}: \ {\mathbb R}^3\times{\mathbb R}^3\to{\mathbb R}^3,
$$
which is bounded and satisfies Lipschitz condition with respect to both variablises. As a space of probability measures we will use the set ${\frak{M}}_2$ of all probability measures on the Borel $\sigma$-field of ${\mathbb R}^3$ equipped with Wasserstain distance of order 2. This distance is defined as follows \cite{21}. For $\mu, \nu\in{\frak{M}}$ consider $C(\mu, \nu)$ a set of all probability distributions on ${\mathbb R}^3\times{\mathbb R}^3$ which have $\mu$ and $\nu$ as its marginals. Then the Wasserstain distance of order 2 between $\mu$ and $\nu$ is
$$
\gamma_2(\mu, \nu)=\inf_{c(\mu, \nu)}\Big(
\iint_{{\mathbb R}^3}\|\vec{u}-\vec{v}\|^2\varkappa(d\vec{u}, d\vec{v})
\Big)^{\frac{1}{2}}.
$$
It is known \cite{21}, that $({\frak{M}}_2, \gamma_2)$ is a complete separable metric space. Now consider a random mapping $\vec{\xi}: {\mathbb R}^3\to{\mathbb R}^3$ which is organized at the  same way as in the previous section (so, we use the same letter for notation). Namely, the coordinates of $\vec{\xi}$ are independent centered Gaussian random fields with the covariance
$$
E\xi_1(\vec{u})\xi_1(\vec{v})=e^{-\|\vec{u}-\vec{v}\|^2}.
$$
For $\mu\in{\frak{M}}_2$ denote by $\mu_\xi$ its image under the mapping $\vec{\xi}$
$$
\mu_\xi=\mu\circ\vec{\xi}^{-1}.
$$
Since $\vec{\xi}$ is infinitely  differentiable random field, then $\mu_\xi$ is a random measure on ${\mathbb R}^3.$ Now consider a function $\vec{a}: {\mathbb R}^3\times{\frak{M}}\to{\mathbb R}^3,$ which is defined as follows
$$
\vec{a}(\vec{u}, \mu)=E\int_{{\mathbb R}^3}h(\vec{u}, \vec{v})\mu_\xi(d\vec{v}).
$$
\begin{lem}
\label{lem2.1}
If $h$ satisfies the Lipschitz condition, then the function $\vec{a}$ satisfies the Lipschitz condition with respect to both variables.
\end{lem}
\begin{proof}
For $\mu_1, \mu_2\in{\frak{M}}_2$ take $\varkappa\in C(\mu_1, \mu_2). $ Then for $u_1, u_2\in {\mathbb R}^3$
$$
\|\vec{a}(\vec{u}_1, \vec{\mu}_1)-\vec{a}(\vec{u}_2, {\mu}_2)\|\leq
$$
$$
\leq E\|\int_{{\mathbb R}^3}h(\vec{u}_1,\vec{v}_1)\mu_{1,\xi}(d\vec{v}_1)-\int_{{\mathbb R}^3}h(\vec{u}_2,\vec{v}_2)\mu_{2,\xi}(\vec{v}_2)\|=
$$
$$
=E\|\int_{{\mathbb R}^3}\int_{{\mathbb R}^3}(h(\vec{u}_1,\vec{v}_1)-h(\vec{u}_2,\vec{v}_2))\varkappa_\xi(d\vec{v}_1, d\vec{v}_2)\|
\leq
$$
$$
\leq
E\|\int_{{\mathbb R}^3}\int_{{\mathbb R}^3}L(\|\vec{u}_1-\vec{u}_2\|+\|\vec{v}_1-\vec{v}_2\|)\varkappa_\xi(d\vec{v}_1, d\vec{v}_2)=
$$
$$
=L\|\vec{u}_1-\vec{u}_2\|+L\int_{{\mathbb R}^3}\int_{{\mathbb R}^3}E\|\vec{\xi}(\vec{v}_1)-\vec{\xi}(\vec{v}_2)\|\varkappa(d\vec{v}_1, d\vec{v}_2)\leq
$$
$$
\leq
L\|\vec{u}_1-\vec{u}_2\|+L\int_{{\mathbb R}^3}\int_{{\mathbb R}^3}
\Big(
E\|\vec{\xi}(\vec{v}_1)-\vec{\xi}(\vec{v}_2)\|^2\Big)^{\frac{1}{2}}\varkappa(d\vec{v}_1, d\vec{v}_2)=
$$
$$
=L\|\vec{u}_1-\vec{u}_2\|+L\int_{{\mathbb R}^3}\int_{{\mathbb R}^3}
\Big(6(1-e^{-\|\vec{v}_1-\vec{v}_2\|^2})\Big)^{\frac{1}{2}}\varkappa(d\vec{v}_1, d\vec{v}_2)\leq
$$
$$
\leq
L\|\vec{u}_1-\vec{u}_2\|+\sqrt{6}L\Big(\int_{{\mathbb R}^3}\int_{{\mathbb R}^3}
\Big(1-e^{-\|\vec{v}_1-\vec{v}_2\|^2}\Big)\varkappa(d\vec{v}_1, d\vec{v}_2)\Big)^{\frac{1}{2}}\leq
$$
$$
\leq
L
\|\vec{u}_1-\vec{u}_2\|+\sqrt{6}L
\Big(
\int_{{\mathbb R}^3}\int_{{\mathbb R}^3}
\|\vec{v}_1-\vec{v}_2\|^2
\varkappa(d\vec{v}_1, d\vec{v}_2)\Big)^{\frac{1}{2}}.
$$

Now taking $\inf$ with respect to $\varkappa\in C(\mu_1,\mu_2)$ one can get the inequality
$$
\|\vec{a}(\vec{u}_1,\mu_1)-\vec{a}(\vec{u}_2,\mu_2)\|\leq L\|\vec{u}_1-\vec{u}_2\|+\sqrt{6}L\gamma_2(\mu_1, \mu_2).
$$
The lemma is proved.
\end{proof}

Consider the following partial case of equation with interaction with the coefficient $\vec{a}$ from the previous lemma.
\begin{equation}
\label{eq2.2}
d\vec{x}(\vec{u}, t)=\vec{a}(\vec{x}(\vec{u},t),\mu_t)dt+\int_{{\mathbb R}^3}b(\vec{x}(\vec{u}, t), \vec{p})\vec{W}(d\vec{p}, dt),
\end{equation}
$$
\vec{x}(\vec{u},0)=\vec{u}, \ \vec{u}\in{\mathbb R}^3, \ \mu_t=\mu_0\circ\vec{x}(\cdot, t)^{-1}.
$$
Suppose, that $\vec{\xi}$ and $\vec{W}$ are independent and $\vec{h}$ has two continuous bounded derivatives on ${\mathbb R}^3\times{\mathbb R}^3.$  Suppose, also, that $b$ has two continuous bounded derivatives as a mapping from ${\mathbb R}^3$ to $L_2({\mathbb R}^3).$ Under such assumptions the solution to the equation with interaction exists and is unique. Namely, the following statement is a partial case of the general theorem from \cite{18}.
\begin{thm}
\label{thm2.2}
Suppose that all formulated above assumptions are satisfied. Then, for the initial mass distribution $\mu_0\in{\frak{M}}_2$ there exists a unique solution $x$ to Cauchy problem \eqref{eq2.2} such that

1) with probability one $\vec{x}$ is continuous with respect to $\vec{u}$ and $t,$ for a fixed $t\geq0 \ \vec{x}(\cdot, t)$ is a diffeomorphism ${\mathbb R}^3$ on itself,

2) for all $t\geq0 \ \mu_t$ is a random element in ${\frak{M}}_2.$
\end{thm}

Now suppose, that $\mu_0$ is a visitation  measure of a closed smooth curve $\gamma_0$ in ${\mathbb R}^3$ which is parametrized by the function $\vec{f}\in C^2_p([0; 2\pi], {\mathbb R}^3).$ Then, due to Theorem \ref{thm2.2} equation \eqref{eq2.2} will define the evolution $\{\gamma_t, t\geq0\}$ of the curve $\gamma_0.$ As  mentioned above all $\gamma_t, t\geq0$ are homotopic and, therefore, have the same topological type. But we will consider curves
$$
\Gamma_t=\vec{\xi}(\gamma_t), \ t\geq0.
$$
As we have already proved they can change its type with  time. $\Gamma_t$ can be considered as a moving random knot in ${\mathbb R}^3.$ In the next section we will consider an example of such stationary motion.

\section{Ornstein--Uhlenbeck random curves process}
\label{section3}
In this section we will consider a partial case of equation \eqref{eq2.2}, which has a stationary solution. Namely, suppose that the motion of $\gamma$ is described by the following equation with interaction
\begin{equation}
\label{eq3.1}
\begin{split}
&
d\vec{x}(\vec{u},t)=A(\vec{x}(\vec{u},t)-\int_{{\mathbb R}^3}\vec{x}(\vec{v},t)\mu_0(d\vec{v}))dt+d\vec{w}(t),\\
&\vec{x}(\vec{u},0)=\vec{u}, \ \vec{u}\in{\mathbb R}^3.
\end{split}
\end{equation}
Here $\mu_0$ is a visitation measure of $\gamma_0.$ As before we suppose that $\gamma_0$ is parametrized by the function $\vec{f}_0\in C^2_p([0; 2\pi], {\mathbb R}^3).$ Then $\mu_0$  has the form
$$
\forall \ \Delta\in{\mathcal B}({\mathbb R}^3): \ \mu_0(\Delta)=\int^{2\pi}_0{1\!\!\,{\rm I}}_\Delta(\vec{f}_0(s))ds.
$$
In (3.1) $\vec{w}$ is a standard Wiener process in ${\mathbb R}^3,$ $A$ is a generator of unitary semi-group in ${\mathbb R}^3.$ It means that the solution to Cauchy problem for matrix-valued equation
$$
\begin{cases}
dU_t=AU_tdt,\\
U_0=I
\end{cases}
$$consists of orthogonal operators, i.e.
$$
\forall \ t\geq0 \ \ U^*_t=U^{-1}_t.
$$
Cauchy problem (3.1) is partial case of \eqref{eq2.2} and, consequently, has a unique solution. Let us clarify the precise form of solution to (3.1). Define for every $t\geq0$ the center of mass for $\mu_t$ as
$$
\vec{m}_t=\int_{{\mathbb R}^3}\vec{v}\mu_t(d\vec{v})=\int_{{\mathbb R}^3}\vec{x}(\vec{u}, t)\mu_0(d\vec{u}).
$$
Then, integrating (3.1) against $\mu_0$ one can get the equation for $\vec{m}_t$
$$
d\vec{m}_t=d\vec{w}(t).
$$
Now, for any $\vec{u}\in{\mathbb R}^3$
$$
d(\vec{x}(\vec{u}, t)-\vec{m}(t))=A(\vec{x}(\vec{u},t)-\vec{m}_t)dt.
$$
Consequently,
$$
\vec{x}(\vec{u},t)-\vec{m}(t)=U_t(\vec{u}-\vec{m}_0).
$$
Hence, the motion of the curve $\gamma$ can be described as follows. Its center of mass is floating over ${\mathbb R}^3$ as a Brownian particle and simultaneously $\gamma$ is rotating around $\vec{m}_t$ accordingly to orthogonal transformation $U_t.$ This character of motion causes the stationarity of  the random knot $\{\Gamma_t; t\geq0\}.$ Let us recall, that 
$$
\Gamma_t=\vec{\xi}(\gamma_t).
$$
Consequently, $\Gamma_t$ is parametrized by the random function $\vec{\xi}(\vec{x}(\vec{f}_0(s), t)),  \ s\in[0; 2\pi].$ Let us shortly denote it by $\vec{\eta}_t.$
\begin{thm}
\label{thm3.1} $\{\eta_t; t\geq0\}$ is a stationary random process  in $C^2_p([0; 2\pi]).$
\end{thm}
\begin{proof}
For $n,m\geq1$ consider sets of numbers $\{s_1, \ldots, s_m\}\subset[0; 2\pi],  \ 0\leq t_1<\ldots<t_n.$ Consider the distribution of the random vector\newline $\vec{\eta}_{t_1}(s_1), \ldots, \vec{\eta}_{t_1}(s_m), \ldots, \vec{\eta}_{t_n}(s_1), \ldots, \vec{\eta}_{t_n}(s_m).$ Conditionally on $\vec{x},$ the distribution of this vector is Gaussian with zero mean and a covariance matrix whose elements are of the form
$$
e^{\|\vec{x}(\vec{f}_0(s_{k_1}),t_{j_1})-\vec{x}(\vec{f}_0(s_{k_2}),t_{j_2})\|^2}.
$$
Let us rewrite
$$
\vec{x}(\vec{f}_0(s_{k_1}),t_{j_1})-\vec{x}(\vec{f}_0(s_{k_2}),t_{j_2})=
$$
$$
=U_{t_{j_1}}(\vec{f}_0(s_{k_1})-\vec{m}_0)-U_{t_{j_2}}(\vec{f}_0(s_{k_2})-\vec{m}_0)+
$$
$$
+\vec{w}(t_{j_1})-\vec{w}(t_{j_2})=U_{t_1}(U_{t_{j_1}-t_1}(\vec{f}_0(s_{k_1})-m_0)-
$$
$$
-U_{t_{j_2}-t_1}(\vec{f}_0(s_{k_2})-m_0)+U^{-1}_{t_1}(\vec{w}(t_{j_1})-\vec{w}(t_{j_2}))
$$
Note, that due to the choice of covariance function the random field $\vec{\xi}$ has the following property
\newline
$\forall \ \vec{u}_1, \ldots, \vec{u}_n\in{\mathbb R}^3$
$$
(\vec{\xi}(\vec{u}_1),\ldots, \vec{\xi}(\vec{u}_n))\overset{d}{=}(\vec{\xi}(U^{-1}_{t_1}\vec{u}_1),\ldots, \vec{\xi}(U^{-1}_{t_1}\vec{u}_n)).
$$
Also note, that $U^{-1}_{t_1}\vec{w}(t), t\geq0$ is a standard Wiener process. Consequently, the distribution of $(\vec{\eta}_{t_1}(s_1), \ldots,\vec{\eta}_{t_n}(s_m) )$ depends only on the differences between $t_1, \ldots, t_n.$ This means stationarity of the considered process. The theorem is proved.
\end{proof}

Let us check that the random knot $\Gamma_t$ changes its topological types. Consider two deterministic knots $G_1, G_2,$which are not of the same type. It is enough to prove, that there exists such $t>0$ that the probability
\begin{equation}
\label{eq3.1}
P\{ \Gamma_0 \ \mbox{has type} \ G_1, \ \Gamma_t \ \mbox{has type} \ G_2 \ \}
\end{equation}
is positive. We will check this by using a rapid decrease of covariance of the random field $\vec{\xi}$ and the independence of $\vec{\xi}$ and $\vec{w}.$

\begin{thm}
\label{thm3.2}
There exists such $t_0,$ for which probability \eqref{eq3.1} is positive.
\end{thm}
\begin{proof}
First, note that for an arbitrary $t\geq0$ and $\vec{u}\in{\mathbb R}^3$
$$
P\{ \vec{\xi}(U_t\gamma_0+\vec{u}) \ \mbox{has type} \ G_2 \ \}=P\{ \vec{\xi}(\gamma_0) \ \mbox{has type} \ G_2 \ \}.
$$
This follows from the form of covariance of $\vec{\xi}.$ Since the covariance of $\vec{\xi}$ decrease to 0 on infinity, there exists such $R>0,$ that\newline
$\forall \ t>0 \ \forall \ \vec{u}\in{\mathbb R}^3, \|\vec{u}\|\geq R$
$$
|P\{\vec{\xi}(\gamma_0) \ \mbox{has type} \ G_1, \vec{\xi}(U_t\gamma_0+\vec{u}) \ \mbox{has type} \ G_2 \ \} -
$$
$$
P\{ \vec{\xi}(\gamma_0) \ \mbox{has type} \ G_1 \}\cdot P\{ \vec{\xi}(U_t\gamma_0+\vec{u}) \ \mbox{has type} \ G_2 \ \}|<
$$
$$
<\frac{1}{2}P\{ \vec{\xi}(\gamma_0) \ \mbox{has type} \ G_1 \ \} \cdot
P\{ \vec{\xi}(U_t\gamma_0+\vec{u}) \ \mbox{has type} \ G_2 \ \}.
$$
Now the statement of the theorem follows from the fact, that in ${\mathbb R}^3$
$$
\|\vec{w}(t)\|\to+\infty, \ t\to\infty \ \mbox{a.s.}
$$
and the independence of $\vec{w}$ and $\vec{\xi}.$ The theorem is proved.
\end{proof}

{\it Remark}. Since it was proved in Section 1 that for fixed $t\geq0$ \ $\Gamma_t$ has no self-intersections with probability one, then the constructed stationary random knot $\{\Gamma_t; t\geq0\}$ changes its type at random moments of time.

\end{document}